\documentclass[11pt,reqno]{amsart}

\RequirePackage[dvipsnames,usenames]{color}

\usepackage{amssymb,amsthm,amsmath,amscd}
\usepackage[backref=page, colorlinks=true, linkcolor=magenta, citecolor=cyan]{hyperref}

\usepackage{mathtools}
\usepackage{soul}

\usepackage{quiver}

\input{kmacros3.sty}

\usepackage{tikz}
\usepackage{tikz-cd}
\usetikzlibrary{cd}
\usepackage{graphicx}
\usepackage[all,cmtip]{xy}
\usepackage{mathpazo}

\usepackage{mabliautoref}
\numberwithin{equation}{theorem}

\usepackage{bm}
\usepackage{pifont}
\usepackage{upgreek}

\usepackage{eucal}
\usepackage{ulem}
\usepackage{stmaryrd}

\usepackage{fullpage}

\usepackage{setspace}

\usepackage{enumerate}
\usepackage{calc}

\usepackage{verbatim}
\usepackage{alltt}

\DeclareMathOperator{\td}{tr. deg.}
\DeclareMathOperator{\Ass}{Ass}

\renewcommand{\O}{\mathcal O}

\begin{document}

\title{Frobenius splitting of valuation rings revisited}
\author{Rankeya Datta}
\email{rankeya.datta@missouri.edu}
\author{Karen E. Smith}
\email{kesmith@umich.edu}

\begin{abstract}
        Let $K$ be a function field over an $F$-finite field $k$ of characteristic $p > 0$ and let $\nu$ be a valuation of $K/k$. We show $\nu$ is an Abhyankar valuation of $K/k$ if and only if the corresponding valuation ring is Frobenius split. The proof proceeds by analyzing when ideals of a valuation ring are uniformly $F$-compatible and establishing a general ramification-theoretic characterization of Frobenius split valuation rings of $F$-finite fields. In addition, when the ground field $k$ is not $F$-finite, we show that the equivalence between Frobenius splitting and the Abhyankar property can fail by constructing an example of an excellent Frobenius split DVR of a function field whose corresponding valuation is not divisorial.  Our methods  do not rely on local uniformization results.
\end{abstract} 

\maketitle

\section{Introduction}

For any valuation ring $V$ of prime characteristic $p$, the Frobenius map is {\it pure}, meaning  that the $p$-th power map on $V$ remains injective after tensoring with any $V$-module \cite[Corollary~3.3]{DattaSmithValuations}. This basic observation  was the starting point of the authors' investigation of various algebraic notions of singularities defined using the Frobenius map, such as $F$-regularity and Frobenius splitting, in the highly non-Noetherian setting of valuation rings \cite{DattaSmithValuations,DattaSmithValuationsErratum}. For valuations of $F$-finite function fields, the following particularly pleasing picture emerged.

\begin{theorem}\cite{DattaSmithValuations, DattaSmithValuationsErratum}
    \label{thm:intro-F-finitness-valuation-rings-function-fields}
  Fix a ground field $k$ of characteristic $p>0$ and assume that $k$ is $F$-finite, i.e., that $k/k^p$ is a finite extension.   Let $K$ be a function field over $k$ and let $(V,\fram_V,\kappa_V)$ be a non-trivial valuation ring of $K/k$. Then the following are equivalent:
    \begin{enumerate}[(i)]
        \item $V$ is $F$-finite.
        \item $V$ is a free $V^p$-module of rank $[K:K^p]$.
        \item $p[\kappa_V \colon \kappa_V^p] = [K : K^p]$.
        \item $V$ is a Frobenius split DVR.
        \item $V$ is an excellent DVR.
        \item $V$ is the valuation ring of a divisorial valuation of $K/k$.
          \item $V$ is split $F$-regular (an adaptation of strong $F$-regularity to non-Noetherian and non-$F$-finite settings).
    \end{enumerate}
\end{theorem}

\autoref{thm:intro-F-finitness-valuation-rings-function-fields} raises the natural question as to when non-Noetherian valuation rings of $F$-finite function fields are Frobenius split.  The purpose of this paper is to answer this question with the following comprehensive result.

\begin{theorem} (see \autoref{thm:Frobenius-splitting-Abhyankar-equivalence})
\label{thm:intro-main-F-split-Abhyankar}
Let $K$ be a finitely generated field extension of an $F$-finite field $k$ of characteristic $p > 0$. Let $\nu$ be a valuation of $K/k$ with valuation ring $(\O_\nu,\fram_\nu,\kappa_\nu)$ and value group $\Gamma_\nu$. Let $\nu^p$ denote the restriction of $\nu$ to the subfield $K^p$ of $K$. Then the following are equivalent:
\begin{enumerate}[(i)]
\item $\O_\nu$ is Frobenius split.\label{thm:intro-main-F-split-Abhyankar.i}
\item $\nu$ is an Abhyankar valuation of $K/k$.\label{thm:intro-main-F-split-Abhyankar.ii}
\item $[\Gamma_\nu : p\Gamma_\nu][\kappa_\nu : \kappa_\nu^p] = [K : K^p]$.\label{thm:intro-main-F-split-Abhyankar.iii}
\item The extension $\nu/\nu^p$ has no ramification defect (\autoref{def:ramification-defect}). \label{thm:intro-main-F-split-Abhyankar.iv}
\end{enumerate}
\end{theorem}

Here,  $\nu$ being \emph{Abhyankar} means that 
\[\dim_{\mathbf Q}(\Gamma_\nu \otimes_{\mathbf Z} \mathbf Q) + \td{\kappa_\nu/k} = \td{K/k},\]
a property that holds for any {\it divisorial} valuation of $K/k$. Thus, Abhyankar valuations can be viewed  as higher (rational) rank analogs of  divisorial valuations. In particular, Abhyankar valuations satisfy many of the desirable finiteness properties of divisorial valuations. For example, their value groups are free abelian groups of finite rank and their residue fields are  finitely generated over the ground field. Moreover, while the outstanding problem of local uniformization of valuations remains open, Abhyankar valuations on function fields \emph{always} admit a local uniformization, i.e., they are centered on a {\it regular} point of a variety $X/k$ with function field $K$ \cite{KnafKuhlmannAbhyankar, CutkoskyAbhyankar}.

Abhyankar valuations are often called \emph{quasi-monomial valuations}, particularly in the characteristic zero setting and when the value group is a subgroup of the real numbers (see {\it e.g.} \cite{JonssonMustataValuations,XuQuasiMonomial}). In this paper,  we will consider Abhyankar valuations that are not necessarily real-valued and, of course, are in the prime characteristic setting. 

  \autoref{thm:intro-F-finitness-valuation-rings-function-fields} and \autoref{thm:intro-main-F-split-Abhyankar} can be viewed as saying that the well-studied notions of $F$-singularities for Noetherian rings, when applied to the setting of valuation rings of function fields, often imply strong finiteness properties. While split $F$-regularity is quite a stringent condition, what is perhaps more surprising (at least to the authors) is that the much weaker property of Frobenius splitting for valuation rings already implies interesting finiteness properties, as captured by the Abhyankar condition.

We also show that the equivalence of Frobenius splitting and the Abhyankar property does not extend to the case where the ground field $k$ fails to be $F$-finite.  Specifically, we construct an example of a Frobenius split excellent DVR of the function field of $\mathbf{P}^2_k$ whose associated valuation is not Abhyankar  (equivalently, not divisorial).  See \autoref{thm:F-split-non-Abhyankar-excellent-DVR-function-field}.

The first author had shown in \cite{DattaFrobeniusSplittingValuations}, under  a separability assumption on the residue field extension, 
that Abhyankar valuations of prime characteristic function fields are always Frobenius split.   \autoref{thm:intro-main-F-split-Abhyankar} removes this assumption and establishes the converse. Furthermore, our proof does not require the use of local uniformization as in that prior work. 

\subsection{The non-function field case} The technical heart of our paper is the following result characterizing Frobenius splitting of valuation rings of $F$-finite fields quite generally  via ramification theory.

\begin{theorem}(see \autoref{thm:F-split-F-finite-no-defect})
    \label{thm:intro-F-split-Abhyankar-center}
        Let $K$ be an $F$-finite field of characteristic $p > 0$. Let $\nu$ be a valuation of $K$ with valuation ring $(\O_\nu, \fram_\nu, \kappa_\nu)$ and value group $\Gamma_\nu$. Then $\O_\nu$ is Frobenius split if and only if 
    \[
    [\Gamma_\nu \colon p\Gamma_\nu][\kappa_\nu \colon \kappa_\nu^p] = [K \colon K^p].
    \]
    Put differently, $\O_\nu$ is Frobenius split if and only if 
  $\nu/\nu^p$ has no ramification defect, where $\nu^p$ denotes the restriction of $\nu$ to the subfield $K^p$ of $K$. 
\end{theorem}

Note that  \autoref{thm:intro-F-split-Abhyankar-center} holds even outside the function field setting. In particular, it implies an analog  of \autoref{thm:intro-main-F-split-Abhyankar} using the concept of an \emph{Abhyankar center} (\autoref{def:Abhyankar-center}), a natural generalization of the notion of an Abhyankar valuation; see    \autoref{cor:F-split-implies-excellent-center-Abhyankar}. 

\medskip
The proof of \autoref{thm:intro-F-split-Abhyankar-center} uses a  characterization of uniformly $F$-compatible ideals of valuation rings that was obtained by the first author in his Ph.D. thesis \cite{DattaThesis};  we reprove this characterization in \autoref{prop:uniformly-F-compatible-valuation-rings} for the reader's benefit.  Our proof also adapts other ideas from \cite{DattaFrobeniusSplittingValuations} to avoid relying on local uniformization results.

We also provide the following ramification-theoretic criterion of $F$-finiteness of valuation rings using the notion of the \emph{initial (ramification) index} of a generically finite extension of valuations (\autoref{def:initial-index}), which should be compared with \autoref{thm:intro-F-split-Abhyankar-center}.

\begin{theorem}(see \autoref{thm:F-finiteness-ramification-theoretic})
    \label{thm:intro-F-finite-ramification-theoretic}
    Let $K$ be an $F$-finite field of characteristic $p > 0$. Let $\nu$ be a valuation of $K$ with valuation ring $(\O_\nu,\fram_\nu,\kappa_\nu)$ and value group $\Gamma_\nu$. Let $\nu^p$ denote the restriction of $\nu$ to the subfield $K^p$ of $K$. Then the following are equivalent:
    \begin{enumerate}[(i)]
        \item $\O_\nu$ is $F$-finite.
        \item $\nu/\nu^p$ has no ramification defect and $\epsilon(\nu|\nu^p) = [\Gamma_\nu\colon p\Gamma_\nu]$, where $\epsilon(\nu|\nu^p)$ is the initial index of $\nu/\nu^p$.
    \end{enumerate}
\end{theorem}

\autoref{thm:intro-F-finite-ramification-theoretic} is a consequence of a ramification-theoretic characterization of finiteness of integral closures of valuation rings in  finite extensions of their ambient fields; see \autoref{thm:finiteness-integral-closure}. However, a finer analysis of the initial index $\epsilon(\nu|\nu^p)$ reveals that $F$-finite valuation rings are often DVRs in algebro-geometric settings; see \autoref{cor:valuations-with-F-finite-Noetherian-centers}. One can view this corollary as a generalization of \autoref{thm:intro-F-finitness-valuation-rings-function-fields} to a non-function field setting.

\subsection{AI Disclosure} The first author had felt for some time that for function fields of $F$-finite surfaces, the only Frobenius split valuation rings should be the ones that are Abhyankar. He posed the surface case as a problem for a Michigan Research Experience for Graduates (MREG), although his group was unsuccessful in solving the problem at the time. Recently, he posed the surface case of the Frobenius splitting question to ChatGPT 6.0 (Astra) Pro, which confirmed his suspicion. The generalization of the surface case of the argument to arbitrary valuations of $F$-finite (function) fields is due to the authors. \autoref{lem:field-theory} and the example in \autoref{sec:non-F-finite-F-split-DVR-function-field} were essentially obtained by ChatGPT 6.0 Pro. However, the authors bear full responsibility for any errors or inaccuracies in the final presentation of these and other results. The paper was written by the authors; in preparing the text, AI was used only as a proofreading tool.

\subsection{Notation and conventions} Given a valuation $\nu$ of a field $K$, we will denote its valuation ring by $\O_\nu$, the maximal ideal of $\O_\nu$ by $\fram_\nu$, and the residue field of $\O_\nu$ by $\kappa_\nu$. The value group of $\nu$ will be denoted by $\Gamma_\nu$. We will frequently extend $\nu$ to a function on all of $K$ by assigning $\nu(0) = \infty$. Thus, $\nu(0) > \nu(x)$ for all $x \in K^\times$. This will make the statements of results such as \autoref{lem:rank-1-lemma} and \autoref{lem:field-theory} cleaner. If $K$ has characteristic $p > 0$, we will usually denote the restriction of $\nu$ to the subfield $K^p$ by $\nu^p$.

Let $I$ be an ideal of a ring $R$. Let $n \in \mathbf{Z}_{\geq 0}$. By $I^{[n]}$ we will denote the ideal generated by the $n$-th powers of elements of $I$, i.e.,
\[
I^{[n]} \coloneq (x^n \mid x \in I).
\]
Let $R$ be a ring of prime characteristic $p > 0$. The $e$-th iterate of the absolute Frobenius endomorphism of $R$ is the ring homomorphism
\begin{align*}
F^e \colon R &\to R\\
r &\mapsto r^{p^e}.
\end{align*}
The target copy of $R$ when viewed as an $R$-module by restriction of scalars along $F^e$ is denoted by $F^e_*R$. When $R$ is a domain, one can identify $F^e$ with the inclusion $R^{p^e} \hookrightarrow R$, where $R^{p^e} \coloneqq \im(F^e)$. In this paper we will frequently switch between $F^e$ and this inclusion.

\section{Preliminaries from valuation theory}

Let $\nu$ be a valuation of a field $K$ with valuation ring $(\O_\nu, \fram_\nu, \kappa_\nu)$. Let $(R,\fram_R,\kappa_R)$ be a local domain. We say $\nu$ is \emph{centered} on $R$ and $R$ is a \emph{center} of $\nu$ if $(\O_\nu, \fram_\nu, \kappa_\nu)$ dominates $(R,\fram_R,\kappa_R)$. In other words, $\nu$ is centered on $R$ if $\nu(x) \geq 0$ for all $x \in R$ and $\nu(x) > 0$ for all $x \in \fram_R$. 

Typically our centers will be Noetherian and their fraction field will equal the ambient field of the valuation. Our starting point is the following fundamental result about Noetherian local centers of a valuation due to Abhyankar.

\begin{theorem}\cite{AbhyankarLocalDomain}
    \label{thm:Abhyankar-inequality}
    Let $(R,\fram_R,\kappa_R)$ be a Noetherian local domain with fraction field $K$. Let $\nu$ be a valuation of $K$ that is centered on $R$. Suppose $(\O_\nu,\fram_\nu,\kappa_\nu)$ is the valuation ring and $\Gamma_\nu$ is the value group of $\nu$. Then
    \begin{equation}
        \label{eq:Abhyankar-inequality}
    \dim_{\mathbf Q}(\Gamma_\nu \otimes_{\mathbf Z} \mathbf Q) + \td{\kappa_\nu/\kappa_R} \leq \dim(R).
    \end{equation}
    If equality holds, then $\Gamma_\nu$ is a finitely generated abelian group and $\kappa_\nu$ is a finitely generated field extension of $\kappa_R$.
\end{theorem}

\begin{definition}
    \label{def:Abhyankar-center}
    Let $(R,\fram_R,\kappa_R)$ be a Noetherian local domain with fraction field $K$. Let $\nu$ be a valuation of $K$ that is centered on $R$. We say that $R$ is an \emph{Abhyankar center} of $\nu$ if equality holds in \autoref{eq:Abhyankar-inequality}.
\end{definition}

In the special case when $R$ is essentially of finite type over a ground field $k$ and $\nu$ is a valuation of $K/k$ centered on $R$, \autoref{thm:Abhyankar-inequality} immediately implies the following:

\begin{corollary}
    \label{cor:Abhyankar-equality-function-fields}
    Let $k$ be a field and $K$ be a finitely generated field extension of $k$. Let $\nu$ be a valuation of $K/k$ with valuation ring $(\O_\nu, \fram_\nu, \kappa_\nu)$ and value group $\Gamma_\nu$. Then
    \begin{equation}
        \label{eq:Abhyankar-equality-function-fields}
        \dim_{\mathbf Q}(\Gamma_\nu \otimes_{\mathbf Z} \mathbf Q) + \td{\kappa_\nu/k} \leq \td{K/k}.
    \end{equation}
    If equality holds, then $\Gamma_\nu$ is a finitely generated abelian group and $\kappa_\nu$ is a finitely generated field extension of $k$.
\end{corollary}

\begin{proof}
    Indeed, let $(R,\fram_R,\kappa_R)$ be a Noetherian local domain essentially of finite type over $k$ with fraction field $K$ such that $\nu$ is centered on $R$. Then $\dim(R) = \td{K/k} - \td{\kappa_R/k}$. By \autoref{thm:Abhyankar-inequality}, we then immediately obtain \autoref{eq:Abhyankar-equality-function-fields}. If equality holds in \autoref{eq:Abhyankar-equality-function-fields}, then equality also holds in \autoref{eq:Abhyankar-inequality} for the essentially of finite type $k$-algebra center $R$, i.e., $R$ is an Abhyankar center of $\nu$. The finite generation of $\kappa_\nu/k$ then follows from the fact that $\kappa_\nu$ is a finitely generated field extension of $\kappa_R$ and $\kappa_R$ is a finitely generated field extension of $k$ by construction.
\end{proof}

\begin{definition}
    \label{def:Abhyankar-valuation-function-field}
    Let $k$ be a field and $K$ be a finitely generated field extension of $k$. Let $\nu$ be a valuation of $K/k$. We say that $\nu$ is an \emph{Abhyankar valuation} of $K/k$ if equality holds in \autoref{eq:Abhyankar-equality-function-fields}.
\end{definition}

\begin{remarks}
    {\*}
    \begin{enumerate}
        \item \autoref{eq:Abhyankar-equality-function-fields} holds under the weaker assumption that $K/k$ has finite transcendence degree. In other words, one does not need $K/k$ to be finitely generated. See \cite[Chapter VI, \S 10.3, Corollary~1]{BourbakiCA}.
        \item A valuation $\nu$ of a function field $K/k$ that admits an Abhyankar center need not be an Abhyankar valuation of $K/k$. For example, one can have a discrete valuation (i.e., $\Gamma_\nu = \mathbf{Z}$) of $k(x,y)/k$ with the property that $\kappa_\nu = k$; see for example the valuation constructed in \autoref{sec:non-F-finite-F-split-DVR-function-field}. Such a valuation is not an Abhyankar valuation of $K/k$ because $\dim_{\mathbf Q}(\Gamma_\nu \otimes_{\mathbf Z} \mathbf Q) + \td{\kappa_\nu/k} = 1 + 0 < 2 = \td{k(x,y)/k}$. However, $\nu$ does admit an Abhyankar center, namely, $\O_\nu$ itself. The key point here is that although $\O_\nu$ is Noetherian, it is not essentially of finite type over $k$. This is why $\nu$ fails to be an Abhyankar valuation of $k(x,y)/k$.
        \item Since the Krull dimension of the valuation ring of $\nu$ is bounded above by the rational rank $\dim_{\mathbf Q}(\Gamma_\nu \otimes_{\mathbf Z} \mathbf Q)$ of $\Gamma_\nu$, we see that valuations of function fields always have \emph{finite rank}, i.e., their valuation rings have finite Krull dimension.
    \end{enumerate}
\end{remarks}



\subsection{Ramification of valuations} We next introduce the notion of ramification defect, following \cite{KuhlmannDefect}. We will only define it in the generality we need.

Let $F/K$ be a finite field extension and $w$ be a valuation of $F$ and $\nu$ denote its restriction to $K$. Suppose $w$ is the unique extension of $\nu$ to $F$ up to equivalence of valuations. If the characteristic of the residue field of $\nu$ is $p > 0$, then the Lemma of Ostrowski \cite[Equation~(2.1)]{KuhlmannDefect} asserts that 
\[
[F:K] = [\Gamma_w\colon \Gamma_\nu][\kappa_w\colon \kappa_\nu] p^{\delta},
\]
for some $\delta \in \mathbf{Z}_{\ge 0}$.

\begin{definition}
    \label{def:ramification-defect}
    In the situation of the previous paragraph, the positive integer $p^\delta$ is called the \emph{ramification defect} of the unique extension of valuations $w/\nu$. If $p^\delta = 1$, i.e., if $\delta = 0$, then we say that $w/\nu$ is \emph{defectless} or \emph{has no ramification defect}.
\end{definition}

Sometimes, the exponent $\delta$ is also called the ramification defect.

\begin{example}
    \label{eg:ramification-defect-char-p}
In this paper we will apply the concept of ramification defect to a very specific setup. Let $K$ be a field of characteristic $p > 0$ and $\nu$ be a valuation of $K$. Let $\nu^p$ denote the restriction of $\nu$ to the subfield $K^p$. Then it is easy to see that $\nu$ is the unique extension of $\nu^p$ to $K$. The valuation ring of $\nu^p$ is $\O_\nu^p$, the value group of $\nu^p$ is $p\Gamma_\nu$, and the residue field of $\nu^p$ is $\kappa_\nu^p$. Thus, when $K$ is $F$-finite, we can talk about the defect of $\nu/\nu^p$, i.e., the positive integer 
\[\frac{[K \colon K^p]}{[\Gamma_\nu\colon \Gamma_{\nu^p}][\kappa_\nu\colon \kappa_{\nu^p}]} = \frac{[K \colon K^p]}{[\Gamma_\nu\colon p\Gamma_{\nu}][\kappa_\nu\colon \kappa_{\nu}^p]}.\]
\end{example}

We also need the following important theorem from the ramification theory of valuations.

\begin{theorem}\cite[Chapter VI, \S 8.3, Theorem 1]{BourbakiCA}
    \label{thm:ram-rel-inequality}
    Let $F/K$ be a finite field extension. Let $\nu$ be a valuation of $K$ and $w_1,\dots,w_n$ be extensions of $\nu$ to $F$ that are pairwise not equivalent. Then
    \[\sum_{i=1}^n [\Gamma_{w_i}\colon \Gamma_\nu][\kappa_{w_i}\colon \kappa_\nu] \le [F:K].\]
\end{theorem}

$[\Gamma_{w_i} \colon \Gamma_\nu]$ is usually called the \emph{ramification index} and $[\kappa_{w_i} \colon \kappa_\nu]$ the \emph{residue degree} of the extension of valuations $w_i/\nu$. Another ramification-theoretic invariant that is less commonly seen is called the \emph{initial ramification index}. Note that \autoref{thm:ram-rel-inequality} implies that there are at most finitely many valuations of $F$ that extend $\nu$.

\begin{definition}
    \label{def:initial-index}
    Let $F/K$ be a finite field extension, $w$ a valuation of $F$ and $\nu$ be the restriction of $w$ to $K$. The \emph{initial ramification index} or just \emph{initial index} of $w/\nu$, denoted $\epsilon(w|\nu)$, is the size of the set $\Sigma \coloneqq \{x \in \Gamma_{w,\geq 0} \colon x < y \text{ for all } y \in \Gamma_{\nu, > 0}\}$.
\end{definition}

Let $x_1, x_2 \in \Sigma$ be distinct elements and assume without loss of generality that $x_1 < x_2$. Then $x_2 - x_1 \notin \Gamma_\nu$. Indeed, otherwise $x_2 - x_1 \in \Gamma_{\nu,>0}$ and $x_2-x_1 \leq x_2$. This would contradict our choice of $x_2$. The upshot is that the classes of $x_1$ and $x_2$ modulo $\Gamma_\nu$ are distinct, which implies that $\epsilon(w|\nu) \leq [\Gamma_w\colon \Gamma_\nu]$. When $\epsilon(w|\nu) = [\Gamma_w\colon \Gamma_\nu]$, every element of $\Gamma_w/\Gamma_\nu$ is the class of some element of $\Gamma_{w,\geq 0}$. 

Our interest in the initial index is because of the following result.

\begin{theorem}\cite[Chapter~VI, \S~8.5, Theorem~2]{BourbakiCA}
    \label{thm:finiteness-integral-closure}
    Let $F/K$ be a finite field extension. Let $\nu$ be a valuation of $K$ with valuation ring $(\O_\nu,\fram_\nu,\kappa_\nu)$. Let $\{w_1,\dots,w_n\}$ be a complete system of extensions of $\nu$ to valuations of $F$. Let $B$ be the integral closure of $\O_\nu$ in $F$. Then the following are equivalent:
    \begin{enumerate}[(i)]
        \item $B$ is module finite over $\O_\nu$.
        \item $\sum_{i=1}^n [\Gamma_{w_i}\colon \Gamma_\nu][\kappa_{w_i}\colon\kappa_\nu] = [F\colon K]$ and for all $i$, $\epsilon(w_i|\nu) = [\Gamma_{w_i}\colon\Gamma_\nu]$.
    \end{enumerate}
\end{theorem}

\section{Background on valuation rings in prime characteristic}

Let $K$ be a field of characteristic $p > 0$ and $\nu$ be a valuation of $K$ with valuation ring $(\O_\nu, \fram_\nu, \kappa_\nu)$. Let $\nu^p$ denote the restriction of $\nu$ to the subfield $K^p$. As seen in \autoref{eg:ramification-defect-char-p}, the valuation ring of $\nu^p$ is $\O_\nu^p$, the residue field of $\nu^p$ is $\kappa_\nu^p$, and the value group of $\nu^p$ is $p\Gamma_\nu$. Since $\O_\nu^p \hookrightarrow \O_\nu$ is an integral extension of domains, it follows that $\O_\nu \otimes_{\O_\nu^p} K^p = K$. Thus, every element of $K$ can be written as a fraction $x/y^p$, where $x, y \in \O_\nu$ and $y \neq 0$.

The first result we recall is a characterization of Abhyankar valuations $\nu$ of $F$-finite function fields in terms of ramification defect of $\nu/\nu^p$.

\begin{theorem}\textup{(\cite[Theorem~5.1]{DattaSmithValuations}, \cite[Theorem~1.1]{DattaSmithValuationsErratum})}
    \label{thm:Abhyankar-valuations-char-p}
    Let $k$ be an $F$-finite field of characteristic $p > 0$. Let $K$ be a finitely generated field extension of $k$ and $\nu$ be a valuation of $K/k$ with valuation ring $(\O_\nu, \fram_\nu, \kappa_\nu)$. Let $\nu^p$ denote the restriction of $\nu$ to the subfield $K^p$. Then the following are equivalent:
    \begin{enumerate}[(i)]
        \item $\nu$ is an Abhyankar valuation of $K/k$.
        \item $[\Gamma_\nu : p\Gamma_\nu][\kappa_\nu : \kappa_\nu^p] = [K : K^p]$.
        \item $\nu/\nu^p$ has no ramification defect.
    \end{enumerate}
\end{theorem}

\autoref{thm:Abhyankar-valuations-char-p} generalizes to a non-function field setting as follows:

\begin{theorem}\cite[Theorem~4.0.3, Corollary~4.0.4]{DattaFrobeniusSplittingValuations}
    \label{thm:Abhyankar-centers-char-p}
    Let $(R,\fram_R,\kappa_R)$ be an $F$-finite Noetherian local domain of characteristic $p > 0$ with fraction field $K$. Let $\nu$ be a valuation of $K$ centered on $R$ with valuation ring $(\O_\nu, \fram_\nu, \kappa_\nu)$. Let $\nu^p$ denote the restriction of $\nu$ to the subfield $K^p$. Then the following are equivalent:
    \begin{enumerate}[(i)]
        \item $R$ is an Abhyankar center of $\nu$.
        \item $[\Gamma_\nu : p\Gamma_\nu][\kappa_\nu : \kappa_\nu^p] = [K : K^p]$.
        \item $\nu/\nu^p$ has no ramification defect.
    \end{enumerate}
    Moreover, if the equivalent conditions hold, then any other excellent (equivalently, $F$-finite) local center of $\nu$ with fraction field $K$ is also an Abhyankar center of $\nu$.
\end{theorem}

\begin{remark}
    \label{rem:F-finite-vs-excellent}
    For a Noetherian domain $R$ of characteristic $p > 0$ whose fraction field $K$ is $F$-finite, $R$ being $F$-finite is equivalent to it being excellent. Indeed, while $F$-finite Noetherian rings are always excellent \cite{KunzOnNoetherianRingsOfCharP}, the converse holds because excellent rings are Japanese (aka N-2) and $R$ is a submodule of the finitely generated integral closure of $R^p$ in $K$.
\end{remark}

We next state a result on Frobenius splitting of Abhyankar valuation rings. This was essentially proved by the first author in \cite{DattaFrobeniusSplittingValuations}.

\begin{theorem}
    \label{thm:Frobenius-split-Abhyankar-valuations}
    Let $K$ be a finitely generated field extension of an $F$-finite field $k$ of characteristic $p > 0$. Let $\nu$ be an Abhyankar valuation of $K/k$ with valuation ring $(\O_\nu, \fram_\nu, \kappa_\nu)$. Then $\O_\nu$ is Frobenius split.
\end{theorem}

\begin{proof}
    This is shown in \cite[\S 3]{DattaFrobeniusSplittingValuations} under the additional assumption that $\kappa_\nu$ is separable over $k$. We now explain how to remove this assumption. Separability of $\kappa_\nu$ over $k$ is used in \cite[\S 3]{DattaFrobeniusSplittingValuations} to apply the local uniformization theorem of Knaf and Kuhlmann \cite[Theorem~1.1]{KnafKuhlmannAbhyankar} for Abhyankar valuations. This gives a regular local center $(A, \fram_A, \kappa_A)$ of $\nu$ that is essentially of finite type over $k$, has fraction field $K$, and satisfies the following two properties:
    \begin{enumerate}[(a)]
        \item $\kappa_A \hookrightarrow \kappa_\nu$ is an isomorphism, 
        \item $A$ has Krull dimension equal to $d \coloneqq \dim_{\mathbf{Q}}(\Gamma_\nu \otimes_{\mathbf{Z}} \mathbf{Q})$ and there exists a regular system of parameters $\{x_1, \ldots, x_d\}$ such that $\{\nu(x_1), \ldots, \nu(x_d)\}$ freely generates $\Gamma_\nu$. 
    \end{enumerate}    
    See \cite[Lemma~3.0.1]{DattaFrobeniusSplittingValuations} for more details. However, the existence of such a regular local center $(A, \fram_A, \kappa_A)$ of $\nu$ is guaranteed by \cite[Theorem~1.1]{CutkoskyAbhyankar} without any separability assumption on $\kappa_\nu$ over $k$. The rest of the proof in \cite[\S 3]{DattaFrobeniusSplittingValuations} goes through without any changes.
\end{proof}

For rank $1$ discrete valuations, one has the following:

\begin{theorem}
    \label{thm:Frobenius-split-discrete-valuations}
    Let $K$ be an $F$-finite field of characteristic $p > 0$. Let $\nu$ be a valuation of $K$ such that the valuation ring $(\O_\nu, \fram_\nu, \kappa_\nu)$ is a DVR. Then the following are equivalent:
    \begin{enumerate}[(i)]
        \item $\O_\nu$ is Frobenius split.
        \item $\O_\nu$ is $F$-finite.
        \item $\O_\nu$ is excellent.
        \item $[\Gamma_\nu : p\Gamma_\nu][\kappa_\nu : \kappa_\nu^p] = [K : K^p]$.
        \item $\nu/\nu^p$ is defectless.
        \item $\dim_{\kappa_\nu^p}(\O_\nu/\fram_\nu^p) = [K : K^p]$.
    \end{enumerate}
    Furthermore, if $K$ is a finitely generated field extension of a field $k$ and $\nu$ is a valuation of $K/k$, then the above conditions are equivalent to $\nu$ being a divisorial valuation of $K/k$.
\end{theorem}

\begin{proof}
    This follows from \cite[Corollary~6.6.3]{DattaSmithValuations} and \cite[\S 5, (10)]{DattaFrobeniusSplittingValuations}. Alternatively, this result also appears in \cite[Corollary~IV.58]{DattaThesis}.
\end{proof}

\subsection{Uniformly \emph{F}-compatible ideals of valuation rings} The results of this subsection appear in the first author's Ph.D. thesis \cite[\S 5.3.2]{DattaThesis}. We will include detailed proofs here for the reader's benefit.

Drawing inspiration from the Noetherian $F$-finite case, one can make the following definition.

\begin{definition}
    \label{def:uniformly-F-compatible-ideal}
    Let $R$ be a ring of prime characteristic $p > 0$ and $\varphi \colon F^e_*R \to R$ be an $R$-linear map. An ideal $I$ of $R$ is \emph{$\varphi$-compatible} if $\varphi(F^e_*I) \subseteq I$. We say $I$ is \emph{uniformly $F$-compatible} if $I$ is $\varphi$-compatible for all $\varphi \in \bigcup_{e \geq 0} \Hom_R(F^e_*R,R)$.
\end{definition}

Thus, $I$ is $\varphi$-compatible precisely when $\varphi$ induces an $R/I$-linear map $\overline{\varphi} \colon F^e_*(R/I) \to R/I$ such that the following diagram commutes:
\[\begin{tikzcd}
	{F^e_*R} & R \\
	{F^e_*(R/I)} & {R/I}
	\arrow["\varphi", from=1-1, to=1-2]
	\arrow["F^e_*\pi"',two heads, from=1-1, to=2-1]
	\arrow["\pi",two heads, from=1-2, to=2-2]
	\arrow["{\overline{\varphi}}"', from=2-1, to=2-2]
\end{tikzcd}.\]
Here the vertical maps are the canonical projections. 

We collect some basic properties of uniformly $F$-compatible ideals.

\begin{lemma}\cite[Lemma~3.1.2]{DattaTuckerSplinterLoci}
    \label{lem:uniformly-F-compatible-properties}
    Let $R$ be a ring of prime characteristic $p > 0$, and let $\Sigma$ denote the collection of uniformly $F$-compatible ideals of $R$. Then we have the following:
    \begin{enumerate}[(1)]
        \item $\Sigma$ is closed under arbitrary sums and arbitrary intersections.
        \item If $I \in \Sigma$ and $\frp \in \Ass_R(R/I)$, then $\frp \in \Sigma$.
        \item If $R$ is Noetherian and $I \in \Sigma$, then $\sqrt{I} \in \Sigma$.
        \item If $R$ is Frobenius split, then every element of $\Sigma$ is a radical ideal.
    \end{enumerate}
\end{lemma}

Since every proper radical ideal of a valuation ring is prime, one immediately obtains:

\begin{corollary}
    \label{cor:F-split-valuation-ring-uniformly-compatible-prime}
    Let $V$ be a valuation ring of characteristic $p > 0$. If $V$ is Frobenius split, then every proper uniformly $F$-compatible ideal of $V$ is a prime ideal.
\end{corollary}

The next result is also straightforward, and we omit its proof.

\begin{lemma}
    \label{lem:intersection-all-Frobenius-powers}
    Let $R$ be a ring of prime characteristic $p > 0$. If $I$ is an ideal of $R$, then
    \[\bigcap_{e \in \mathbf{N}} I^{[p^e]}\]
    is uniformly $F$-compatible.
\end{lemma}

Recall that for an ideal $I$ of a ring $R$ and $n \in \mathbf{N}$, $I^{[n]} = (i^n \colon i \in I)$ is the ideal generated by the $n$-th powers of elements of $I$. This notation is most frequently used when $R$ has prime characteristic $p > 0$ and $n = p^e$, because then $I^{[p^e]}$ is just the expansion of $I$ under the $e$-th iterate of Frobenius $F^e \colon R \to F^e_*R$. However, the notation makes sense for an ideal of a ring of arbitrary characteristic and any non-negative integer power. Note that when $I$ is a principal ideal, $I^{[n]} = I^n$.

\begin{lemma}
    \label{lem:intersection-ideal-powers-valuation-ring}
    Let $(V,\fram_V,\kappa_V)$ be a valuation ring and $I \subsetneq V$ be a proper ideal of $V$. Then we have the following:
    \begin{enumerate}[(i)]
        \item For all $n \geq 0$, $I^{[n]} = I^n$.\label{lem:intersection-ideal-powers-valuation-ring.i}
        \item $\bigcap_{n \in \mathbf{N}} I^n$ is a prime ideal of $V$.\label{lem:intersection-ideal-powers-valuation-ring.ii}
        \item If $\fram_V$ is not finitely generated, then $\fram_V = \fram_V^2$.\label{lem:intersection-ideal-powers-valuation-ring.iii}
        \item If $\fram_V$ is a nonzero finitely generated ideal and $\mathfrak{M} \coloneqq \bigcap_{n \in \mathbf{N}} \fram_V^n$, then $V/\mathfrak{M}$ is a DVR.\label{lem:intersection-ideal-powers-valuation-ring.iv}
    \end{enumerate}
\end{lemma}

\begin{proof}
    \autoref{lem:intersection-ideal-powers-valuation-ring.i} It is enough to show $I^n \subseteq I^{[n]}$. If $i \in I^n$, then clearly there exists a finitely generated ideal $J \subset I$ such that $i \in J^n$. But $J$ must be principal, and so $i \in J^n = J^{[n]} \subset I^{[n]}$.

    \autoref{lem:intersection-ideal-powers-valuation-ring.ii} Let $\mathcal{I} \coloneqq \bigcap_{n \in \mathbf{N}} I^n$. Since $I$ is a proper ideal, so is $\mathcal{I}$. It is enough to show that $\mathcal{I}$ is radical. Let $x \in V$ be such that $x^m \in \mathcal{I}$ for some $m \geq 1$. Then for all $n \in \mathbf{N}$, $x^m \in (I^n)^m$. Hence, for all $n \in \mathbf{N}$, there exists a finitely generated, hence principal, ideal $J \subset I^n$ such that $x^m \in J^m$. Hence, there exists $i \in I^n$ such that $i^m | x^m$. Then $i | x$ because $m\nu(i) \leq m\nu(x)$ implies $\nu(i) \leq \nu(x)$. Consequently, $x \in I^n$ for all $n \in \mathbf{N}$, that is, $x \in \mathcal{I}$.

    \autoref{lem:intersection-ideal-powers-valuation-ring.iii} Since $\fram_V$ is not finitely generated, $\fram_V \neq (0)$. Let $x \in \fram_V \setminus \{0\}$. Let $\Gamma$ be the value group of the associated valuation $\nu$. Since $\fram_V$ is not finitely generated, there exists $\alpha \in \Gamma$ such that $0 < \alpha < \nu(x)$. Otherwise, $\nu(x)$ would be the smallest element of $\Gamma$ that is $> 0$, forcing $\fram_V$ to be generated by $x$. For a similar reason, there must exist $\beta \in \Gamma$ such that $0 < \beta < \min\{\alpha,\nu(x) -\alpha\}$. Let $y \in \fram_V$ be such that $\nu(y) = \beta$. Then $\nu(y^2) = 2\beta < \alpha + \nu(x) - \alpha = \nu(x)$, and so $y^2|x$. Then $x \in \fram_V^2$. Since $x$ was an arbitrary non-zero element of $\fram_V$, this shows $\fram_V = \fram_V^2$.

    \autoref{lem:intersection-ideal-powers-valuation-ring.iv} This was observed in \cite[\S 6.5]{DattaSmithValuations} for characteristic $p > 0$ valuation rings. The proof, however, is a simple consequence of the previous parts.  By \autoref{lem:intersection-ideal-powers-valuation-ring.ii} we know $\mathfrak M$ is a prime ideal of $V$, so $V/\mathfrak{M}$ is a valuation ring. Also, $\mathfrak M \neq \fram_V$ because $\fram_V$ is finitely generated and hence principal. Therefore, $V/\mathfrak M$ has a principal maximal ideal. Let $\frq \in \Spec(V)$. If $\frq \neq \fram_V$, then for all $n \in \mathbb{N}$, $\frq \subseteq \fram_V^n$ because any two ideals of a valuation ring are comparable. Consequently, $\frq \subseteq \mathfrak M$. Thus, $\Spec(V/\mathfrak M)$ has only two prime ideals, and so, it is a valuation ring of Krull dimension $1$ with a principal maximal ideal, i.e., it is a DVR.
\end{proof}

We are now ready to prove the main result of this subsection.

\begin{proposition}
    \label{prop:uniformly-F-compatible-valuation-rings}
Let $(V,\fram_V,\kappa_V)$ be a valuation ring of prime characteristic $p > 0$. Then any $\frp \in \Spec(V)$ such that $\frp \neq \fram_V$ is uniformly $F$-compatible. In addition, if $\fram_V$ is not finitely generated, then $\fram_V$ is also uniformly $F$-compatible.
\end{proposition}

\begin{proof}
    If $\fram_V$ is not finitely generated, then $\fram_V = \bigcap_{e \in \mathbf{N}} \fram_V^{[p^e]}$ by \autoref{lem:intersection-ideal-powers-valuation-ring}. Hence, $\fram_V$ is uniformly $F$-compatible by \autoref{lem:intersection-all-Frobenius-powers}.

    Now suppose $\frp$ is a non-maximal prime ideal of $V$. Define 
    \[\Sigma \coloneqq \{a \in V \colon a \in \fram_V \setminus \frp\}.\]
    Note $\Sigma \neq \emptyset$ because $\frp \neq \fram_V$. It is enough to show that 
    \begin{equation}
        \label{eq:uniformly-F-compatible}
    \frp = \bigcap_{a \in \Sigma} \bigcap_{e \in \mathbf{N}} (a^{p^e}).
    \end{equation}
    Indeed, $\bigcap_{e \in \mathbf{N}} (a^{p^e})$ is uniformly $F$-compatible by \autoref{lem:intersection-all-Frobenius-powers}, while arbitrary intersections of uniformly $F$-compatible ideals are uniformly $F$-compatible by \autoref{lem:uniformly-F-compatible-properties}.
    If $a \notin \frp$, then $a^{p^e} \notin \frp$ for all $e \in \mathbf{N}$. Thus, $\frp \subset (a^{p^e})$ for all $e \in \mathbf{N}$ by comparability of ideals of a valuation ring. This shows that
    \[\frp \subseteq \bigcap_{a \in \Sigma} \bigcap_{e \in \mathbf{N}} (a^{p^e}).\]
    Note that for all $a \in \Sigma$, $\bigcap_{e \in \mathbf{N}} (a^{p^e})$ is a prime ideal by \autoref{lem:intersection-ideal-powers-valuation-ring}. Since $\Spec(V)$ is totally ordered by inclusion, $\bigcap_{a \in \Sigma} \bigcap_{e \in \mathbf{N}} (a^{p^e})$ is also prime. Thus, to show \autoref{eq:uniformly-F-compatible}, it is enough to show that if $\frq \in \Spec(V)$ satisfies $\frp \subsetneq \frq$, then $\frq \neq \bigcap_{a \in \Sigma} \bigcap_{e \in \mathbf{N}} (a^{p^e})$. By hypothesis, there exists $a \in \frq\setminus\frp$. Thus, $a \in \Sigma$, and it is enough to show that $\frq \nsubseteq \bigcap_{e \in \mathbf{N}} (a^{p^e})$, or equivalently, that $\frq \neq \bigcap_{e \in \mathbf{N}} (a^{p^e})$ (the equivalence holds because $a \in \frq$). If $\frq = \bigcap_{e \in \mathbf{N}} (a^{p^e})$, then for all $e \in \mathbf{N}$, $(a^{p^e}) \subseteq \frq \subseteq (a^{p^e})$, i.e., $\frq = (a^{p^e})$. Hence, for all $e \in \mathbf{N}$, $(a^{p^e}) = (a^{p^0}) = (a)$. But this is impossible because $V$ is a domain and $a$ is not a unit.
\end{proof}

As a consequence of the definition of uniform $F$-compatibility, we then have:

\begin{corollary}
    \label{cor:Frobenius-splitting-compatible}
    Let $(V,\fram_V,\kappa_V)$ be a valuation ring of characteristic $p > 0$. Suppose $V$ is Frobenius split. Then for every non-maximal prime ideal $\frp$ of $V$, $V/\frp$ is also a Frobenius split valuation ring.
\end{corollary}

\begin{proof}
    Fix a Frobenius splitting $\varphi \colon F_*V \to V$. Since $\frp$ is uniformly $F$-compatible by \autoref{prop:uniformly-F-compatible-valuation-rings}, we get an induced Frobenius splitting $\overline{\varphi} \colon F_*(V/\frp) \to V/\frp$. The fact that $V/\frp$ is a valuation ring is well known.
\end{proof}

\section{A ramification-theoretic characterization of Frobenius splitting}

Our goal is to show the following result.

\begin{theorem}
    \label{thm:Frobenius-splitting-Abhyankar-equivalence}
    Let $k$ be an $F$-finite field of characteristic $p > 0$ and $K$ be a finitely generated field extension of $k$. Let $\nu$ be a valuation of $K/k$ with valuation ring $(\O_\nu,\fram_\nu,\kappa_\nu)$ and value group $\Gamma_\nu$. Let $\nu^p$ denote the restriction of $\nu$ to the subfield $K^p$. Then the following are equivalent:
    \begin{enumerate}[(i)]
        \item $\O_\nu$ is Frobenius split.\label{thm:Frobenius-splitting-Abhyankar-equivalence.i}
        \item $\nu$ is an Abhyankar valuation of $K/k$.\label{thm:Frobenius-splitting-Abhyankar-equivalence.ii}
        \item $[\Gamma_\nu\colon p\Gamma_\nu][\kappa_\nu\colon\kappa_\nu^p] = [K\colon K^p]$.\label{thm:Frobenius-splitting-Abhyankar-equivalence.iii}
        \item $\nu/\nu^p$ has no ramification defect.\label{thm:Frobenius-splitting-Abhyankar-equivalence.iv}
    \end{enumerate}
\end{theorem}

We have \autoref{thm:Frobenius-splitting-Abhyankar-equivalence.i}$\impliedby$ \autoref{thm:Frobenius-splitting-Abhyankar-equivalence.ii}$\Longleftrightarrow$ \autoref{thm:Frobenius-splitting-Abhyankar-equivalence.iii} $\Longleftrightarrow$ \autoref{thm:Frobenius-splitting-Abhyankar-equivalence.iv} by \autoref{thm:Frobenius-split-Abhyankar-valuations} and \autoref{thm:Abhyankar-valuations-char-p}. While it suffices to show \autoref{thm:Frobenius-splitting-Abhyankar-equivalence.i}$\implies$ \autoref{thm:Frobenius-splitting-Abhyankar-equivalence.ii}, we will actually show that \autoref{thm:Frobenius-splitting-Abhyankar-equivalence.i} is equivalent to \autoref{thm:Frobenius-splitting-Abhyankar-equivalence.iii} in a non-function field setting as well. This key result is established in \autoref{thm:F-split-F-finite-no-defect}. Another consequence of \autoref{thm:F-split-F-finite-no-defect} is that one now has a purely valuation-theoretic proof of \autoref{thm:Frobenius-splitting-Abhyankar-equivalence} that does not use the local uniformization of Abhyankar valuations \cite{KnafKuhlmannAbhyankar,CutkoskyAbhyankar}.

The following lemma draws inspiration from \cite[Lemma~3.0.4]{DattaFrobeniusSplittingValuations}. Recall our convention that if $\nu$ is a valuation of a field $K$, then we extend it to a function on all of $K$ by assigning $\nu(0) = \infty$. Thus, $\nu(0) > \nu(x)$ for all $x \in K^\times$.

\begin{lemma}
    \label{lem:rank-1-lemma}
    Let $K$ be a field of characteristic $p > 0$. Let $\nu$ be a valuation of $K$ with valuation ring $(\O_\nu,\fram_\nu,\kappa_\nu)$. Assume the following two conditions hold:
    \begin{enumerate}[(a)]
        \item $\fram_\nu$ is not a finitely generated ideal, and
        \item $\O_\nu$ is Frobenius split.
    \end{enumerate}
    If $\sigma \colon \O_\nu \to \O_\nu^p$ is a Frobenius splitting, then for all $x \in \O_\nu \setminus \{0\}$,
    \[\nu(\sigma(x)) \geq \nu(x).\]
\end{lemma}

\begin{proof}
    Since $\fram_\nu$ is not finitely generated, \autoref{prop:uniformly-F-compatible-valuation-rings} implies that $\fram_\nu$ is uniformly $F$-compatible. That is, $\sigma(\fram_\nu) \subseteq \fram_\nu \cap \O_\nu^p$. Assume there exists $x \in \O_\nu \setminus \{0\}$ such that $\nu(\sigma(x)) < \nu(x)$. Then $\sigma(x) \neq 0$ and
    \[u \coloneqq \frac{x}{\sigma(x)} \in \fram_\nu.\]
    Hence $\sigma(x) = \sigma(\sigma(x)u) = \sigma(x)\sigma(u)$, where the second equality follows because $\sigma(x) \in \O_\nu^p$ and $\sigma$ is $\O_\nu^p$-linear. Since $\sigma(x) \neq 0$, we get $\sigma(u)  = 1$, which is a contradiction because $u \in \fram_\nu$ and $\sigma(\fram_\nu) \subseteq \fram_\nu \cap \O_\nu^p$.
\end{proof}

\begin{lemma}
    \label{lem:field-theory}
    Let $F/K$ be a finite field extension. Let $w$ be a valuation of $F$ and $\nu$ denote the restriction of $w$ to $K$. Let $(\O_w, \fram_w, \kappa_w)$ be the valuation ring of $w$ and $(\O_\nu, \fram_\nu, \kappa_\nu)$ be the valuation ring of $\nu$. Suppose there exists a $K$-linear map $\sigma \colon F \to K$ such that 
    \begin{enumerate}[(1)]
        \item $\sigma(1) = 1$, and
        \item $\nu(\sigma(x)) = w(\sigma(x)) \geq w(x)$ for all $x \in F \setminus \{0\}$.
    \end{enumerate}
Then $w$ is the unique extension (up to equivalence) of $\nu$ to $F$ and we have $[F \colon K] = [\Gamma_w \colon \Gamma_\nu][\kappa_w \colon \kappa_\nu]$.
\end{lemma}

\begin{proof}
    The uniqueness of $w$ follows from the equality $[F \colon K] = [\Gamma_w \colon \Gamma_\nu][\kappa_w \colon \kappa_\nu]$ and \autoref{thm:ram-rel-inequality}. We therefore prove this equality. We also have $[\Gamma_w \colon \Gamma_\nu][\kappa_w \colon \kappa_\nu] \leq [F \colon K]$ by \autoref{thm:ram-rel-inequality}. Let $e \coloneq [\Gamma_w \colon \Gamma_\nu]$ and $f \coloneq [\kappa_w \colon \kappa_\nu]$. It is enough to show that we have an injective $K$-linear map $F \to K^{\oplus ef}$. Let $\{x_1, \ldots, x_e\} \subseteq F \setminus \{0\}$ be such that $\{w(x_1), \ldots, w(x_e)\}$ is a complete set of representatives of the cosets of $\Gamma_\nu$ in $\Gamma_w$. Let $\{y_1, \ldots, y_f\} \subseteq \O_w^\times$ be such that $\{\overline{y_1}, \ldots, \overline{y_f}\}$ is a basis of $\kappa_w$ over $\kappa_\nu$, where $\overline{y_i}$ denotes the image of $y_i$ in $\kappa_w$. Consider the $K$-linear map
    \begin{align*}
    \Phi \colon F &\to K^{\oplus ef}\\
    \alpha &\mapsto (\sigma(x_iy_j\alpha))_{1 \leq i \leq e, 1 \leq j \leq f}.
    \end{align*}
    We will show $\Phi$ is injective. Suppose $\alpha \in F \setminus \{0\}$. Choose $x_i$ such that $w(x_i) + \Gamma_\nu = w (\alpha^{-1}) + \Gamma_\nu$. Then $w(x_i\alpha) \in \Gamma_\nu$. Choose $\beta \in K \setminus \{0\}$ such that $w(x_i\alpha) = \nu(\beta) = w(\beta)$. Then $c \coloneqq (x_i\alpha)/\beta \in \O_w^\times$. Let $\overline{c}$ denote the image of $c$ in $\kappa_w$. Since multiplication by $\overline{c}$ is a $\kappa_\nu$-linear automorphism of $\kappa_w$, $\{\overline{cy_i} \colon i = 1, \ldots, f\}$ is a basis of $\kappa_w$ over $\kappa_\nu$. 

    Condition (2) implies that $\sigma(\O_w) \subseteq \O_\nu$ and $\sigma(\fram_w) \subseteq \fram_\nu$. Thus, $\sigma$ induces a $\kappa_\nu$-linear splitting $\overline{\sigma} \colon \kappa_w \to \kappa_\nu$ of $\kappa_\nu \hookrightarrow \kappa_w$. In particular, $\overline{\sigma}$ is non-zero. Thus, there exists $j$ such that $\overline{\sigma}(\overline{cy_j}) \neq 0$. Then $\sigma(cy_j) \neq 0$, and hence $\sigma(x_iy_j\alpha) = \sigma(\beta cy_j) = \beta\sigma(cy_j) \neq 0$. Hence $\Phi(\alpha) \neq 0$, and we conclude that $\Phi$ is injective.
\end{proof}

Let $K$ be an $F$-finite field of characteristic $p > 0$ and $\nu$ be a valuation of $K$ and $\nu^p$ denote the restriction of $\nu$ to $K^p$. Note that $\nu$ is the unique extension of $\nu^p$ to $K$ and we have by \autoref{thm:ram-rel-inequality} that
\[
[\Gamma_\nu \colon p\Gamma_\nu][\kappa_\nu \colon \kappa_\nu^p] \leq [K \colon K^p].
\]
The next lemma provides a finer analysis of the extension $\nu/\nu^p$ when equality holds in the above inequality, that is, when $\nu/\nu^p$ does not have ramification defect. The result can be compared with the proof of \cite[Claim~3.0.3]{DattaFrobeniusSplittingValuations}, which shows something similar under the assumption that $\nu$ admits a particularly nice local uniformization. The lemma below does not require such an assumption.

\begin{lemma}
    \label{lem:valuation-minimizer-basis}
    Let $K$ be an $F$-finite field of characteristic $p > 0$ and $\nu$ be a valuation of $K$ with restriction $\nu^p$ to $K^p$. Let $(\O_\nu, \fram_\nu, \kappa_\nu)$ be the valuation ring of $\nu$. Suppose that $\nu/\nu^p$ is defectless, i.e., $[\Gamma_\nu \colon p\Gamma_\nu][\kappa_\nu \colon \kappa_\nu^p] = [K \colon K^p]$. Fix $a_1, \dots, a_e\in K\setminus\{0\}$ so that $\nu(a_1), \dots, \nu(a_e)$ are representatives for the distinct cosets  of  $\Gamma_\nu/p\Gamma_\nu$. Fix $b_1, \dots, b_f \in \O_\nu^\times$ so that their images in $\kappa_\nu$ are a basis for $\kappa_\nu$ over $\kappa_\nu^p$. Then for all $1\leq i \leq e$, $1\leq j \leq f$ and $c_{ij}^p\in K^p$, we have
\begin{equation}
    \label{eq:valuation-minimizer}
  \nu \left(\sum_{i,j}c_{ij}^pa_i b_j\right)
= \min\{\nu(c_{ij}^pa_i b_j)\}_{i,j}.
\end{equation}
Consequently, $\{a_ib_j\}_{i,j}$ forms a basis for $K$ over $K^p$.
\end{lemma}

\begin{proof}
The final assertion about $\{a_ib_j\}_{i,j}$ forming a basis for $K$ over $K^p$ follows immediately from \autoref{eq:valuation-minimizer} because the latter implies that $\{a_ib_j\}_{i,j}$ is a linearly independent subset of $K$ over $K^p$ of size $ef = [K : K^p]$. Hence, it remains to establish \autoref{eq:valuation-minimizer}.

Let $x=\sum_{i,j}c_{ij}^pa_i b_j$, where $c_{ij}\in K$.    Set 
$\alpha=\min\{\nu(c_{ij}^{p}a_i b_j)\}_{i,j}.$ We want to show that $\nu(x)=\alpha$. We may assume that not all the $c_{ij}$ are zero as otherwise $\nu(x) = \nu(c_{ij}^pa_ib_j) = \infty$ for all $i, j$ and there is nothing to show.

Choose a pair of indices 
$(i_0,j_0)$  attaining the minimum $\alpha$. 
Write $x=  y + z$ where 
\[
 y=\sum_j c_{i_0j}^{p}a_{i_0}b_j, \qquad z=\sum_{i\ne i_0,j}c_{ij}^{p}a_i b_j. \]
 Furthermore, for both $y$ and $z$ we only consider those terms of the sum for which $c_{ij}^p \neq 0$. We will show that
\begin{enumerate}[(a)]
    \item $\nu(z)>\alpha$, and
    \item $\nu(y)=\alpha$,
\end{enumerate}
 from which it will follow that $\nu(x) = \alpha$, establishing the lemma.

 For (a), since each $b_j$ is a unit in $\O_\nu$, we have
\[
\nu(c_{ij}^pa_i b_j)
=p\nu(c_{ij})+\nu(a_i).
\]
This means that $\nu(c_{ij}^pa_i b_j)$  is in the same $p\Gamma_\nu$ equivalence class as $\nu(a_i)$ but in a different equivalence class from $ \nu(c_{k\ell}^pa_kb_\ell)$ if $i\neq k$ and $c_{ij}^p \neq 0 \neq c_{k\ell}^p$, since the $p\Gamma_\nu$  classes of $\nu(a_i)$ and $\nu(a_k)$ are distinct by our choice of $a_i$'s. Thus 
if $i\neq k$ and $c_{ij}^p \neq 0 \neq c_{k\ell}^p$, we have 
\begin{equation}
    \label{eq:distinct-valuation-classes}
\nu(c_{ij}^pa_i b_j) \neq  \nu(c_{k\ell}^pa_k b_\ell),
\end{equation} 
and so, every nonzero term of $z$  has valuation strictly greater than $\nu(c_{i_0j_0}^pa_{i_0}b_{j_0}) =\alpha$, and so $\nu(z)>\alpha$.

For (b), set $d_j \coloneqq \frac{c_{i_0j}}{c_{i_0j_0}}$, for each $j=1, \dots, f$, and observe that 
 $\nu (d_j)\geq 0$ by our choice of $(i_0, j_0)$.  Now $\nu(y) $ is
\begin{equation}
    \label{D}
\nu(y) =
\nu\left(\sum_{j=1}^fc_{i_0j_0}^pd_j^pa_{i_0} b_j\right) = 
\nu\left(c_{i_0j_0}^p a_{i_0}\sum_{j=1}^fd_j^p b_j\right) = \nu(a_{i_0}c_{i_0j_0}^p) + \nu\left(\sum_{j=1}^fd_j^pb_j\right).
\end{equation}
Note that all $d_j$ and $b_j$ are in $\O_\nu$. So computing the residue of $\sum_{j=1}^f d_j^pb_j$ modulo $\fram_\nu$, we have 
\begin{equation}\label{2}
\overline{\sum_j d_j^pb_j}
=\sum_{j=1}^f\overline{d_j}^p\,\overline{b_j} \neq 0
\end{equation}
since the $\overline{b_j}$ were chosen to be linearly independent over $\kappa_\nu^p$ and $d_{j_0}^p= \left( \frac{c_{i_0j_0}}{c_{i_0j_0}}\right)^p =1$.
So $\sum_{j=1}^fd_j^pb_j\in \O_\nu^\times$, and its value is zero. Glancing back at \autoref{D}, we conclude that $\nu(y) = \alpha$. This completes the proof of the lemma.
\end{proof}

The next result is the technical heart of this paper and provides a ramification-theoretic characterization of Frobenius split valuation rings of $F$-finite fields.

\begin{theorem}
    \label{thm:F-split-F-finite-no-defect}
    Let $K$ be an $F$-finite field of characteristic $p > 0$. Let $\nu$ be a valuation of $K$ with valuation ring $(\O_\nu, \fram_\nu, \kappa_\nu)$ and value group $\Gamma_\nu$. Then $\O_\nu$ is Frobenius split if and only if $\nu/\nu^p$ is defectless, i.e., 
    \begin{equation}
        \label{eq:F-split-implies-no-ramification}
    [\Gamma_\nu \colon p\Gamma_\nu][\kappa_\nu \colon \kappa_\nu^p] = [K \colon K^p].
    \end{equation}
\end{theorem}
 
\begin{proof}
    We begin with the forward implication, that is, we begin by showing that if $\O_\nu$ is Frobenius split then $\nu/\nu^p$ is defectless. First suppose $\fram_\nu$ is not finitely generated. Then by \autoref{lem:rank-1-lemma}, if $\sigma \colon \O_\nu \to \O_\nu^p$ is a Frobenius splitting, we have $\nu(\sigma(x)) \geq \nu(x)$ for all $x \in \O_\nu \setminus \{0\}$. Extend $\sigma$ to a Frobenius splitting $K \to K^p$, which we also denote by $\sigma$. Note that $\O_\nu \otimes_{\O_\nu^p} K^p = K$. Thus, every element of $K\setminus\{0\}$ can be written as a fraction of the form $x/y^p$ with $x \in \O_\nu$ and $y \in \O_\nu \setminus \{0\}$. By $K^p$-linearity of $\sigma$, we have
    \[\nu(\sigma(x/y^p)) = \nu(\sigma(x)/y^p) = \nu(\sigma(x)) - p\nu(y) \geq \nu(x) - p\nu(y) = \nu(x/y^p).\]
    Then by \autoref{lem:field-theory} applied to the finite field extension $K/K^p$ and the valuation $\nu$ of $K$ with restriction $\nu^p$ to $K^p$, we get $[\Gamma_\nu : p\Gamma_\nu][\kappa_\nu : \kappa_\nu^p] = [K : K^p]$.

    In the general case we will proceed by induction on $[\Gamma_\nu \colon p\Gamma_\nu]$, which is finite because by \autoref{thm:ram-rel-inequality} we have $[\Gamma_\nu \colon p\Gamma_\nu][\kappa_\nu : \kappa_\nu^p] \leq [K \colon K^p]$.

    Suppose $[\Gamma_\nu \colon p\Gamma_\nu] = 1$. We may assume $\Gamma_\nu \neq 0$ as otherwise $\nu$ is the trivial valuation with valuation ring $K$, and \autoref{eq:F-split-implies-no-ramification} is automatically satisfied. If $\Gamma_\nu \neq 0$ and $\Gamma_\nu = p\Gamma_\nu$, then $\fram_\nu$ cannot be finitely generated and \autoref{eq:F-split-implies-no-ramification} holds from the case shown above.

    Let $[\Gamma_\nu \colon p\Gamma_\nu] > 1$. We may again assume that $\fram_\nu$ is finitely generated, that is, it is principal. Let 
    \[\frp \coloneqq \bigcap_{n \in \mathbf N} \fram_\nu^n.\]
    Then by \autoref{lem:intersection-ideal-powers-valuation-ring}~\autoref{lem:intersection-ideal-powers-valuation-ring.iv}, $\O_\nu/\frp$ is a DVR. Let $w$ be the valuation of $K$ whose valuation ring is $(\O_\nu)_\frp$ and $\overline{w}$ be the valuation of $\kappa(\frp)$ whose valuation ring is $\O_\nu/\frp$. Note that $\nu$ is the compositum of $w$ and ${\overline{w}}$, and we have a short exact sequence of value groups (see \cite[Proposition~1.11]{VaquieValuations})
    \begin{equation}
        \label{eq:value-groups-ses}
        0 \longrightarrow \Gamma_{\overline{w}} \longrightarrow \Gamma_\nu \longrightarrow \Gamma_w \longrightarrow 0.
    \end{equation}
    We have seen that $\Gamma_{\overline w} \cong \mathbf{Z}$. Also, since $\Gamma_w$ is an ordered abelian group, it is torsion-free and hence a flat $\mathbf{Z}$-module. Therefore the above short exact sequence of value groups is universally exact (\cite[\href{https://stacks.math.columbia.edu/tag/058I}{Tag~058I}, \href{https://stacks.math.columbia.edu/tag/058M}{Tag~058M}]{stacks-project}), or equivalently, $\Gamma_{\overline w} \to \Gamma_\nu$ is a pure map of $\mathbf{Z}$-modules. Thus, 
    \[
    0 \to \Gamma_{\overline{w}}/p\Gamma_{\overline{w}} \to \Gamma_\nu/p\Gamma_\nu \to \Gamma_w/p\Gamma_w \to 0
    \]
    is exact as well. Hence,
    \begin{equation}
        \label{eq:p-index}
        [\Gamma_\nu \colon p\Gamma_\nu] = [\Gamma_{\overline{w}} \colon p\Gamma_{\overline{w}}][\Gamma_w \colon p\Gamma_w] = p[\Gamma_w \colon p\Gamma_w].
    \end{equation}
    In particular, we see that $[\Gamma_w \colon p\Gamma_w] < [\Gamma_\nu \colon p\Gamma_\nu]$. 

    Since $\frp$ is a non-maximal prime ideal of the Frobenius split valuation ring $\O_\nu$, the valuation ring $\O_\nu/\frp$ is Frobenius split by \autoref{cor:Frobenius-splitting-compatible}. Note that $\kappa(\frp)$ is $F$-finite because all valuation rings of $K$ have $F$-finite residue fields by \autoref{thm:ram-rel-inequality} and $\kappa(\frp)$ is the residue field of the valuation ring $(\O_\nu)_\frp$ of $K$. Then by \autoref{thm:Frobenius-split-discrete-valuations} applied to the DVR $\O_\nu/\frp$ with $F$-finite fraction field $\kappa(\frp)$, we get
    \begin{equation}
        \label{eq:rank-1-step}
        [\Gamma_{\overline{w}} \colon p\Gamma_{\overline{w}}][\kappa_\nu \colon \kappa_\nu^p] = [\Gamma_{\overline{w}} \colon p\Gamma_{\overline{w}}][\kappa_{\overline w} \colon \kappa_{\overline w}^p] = [\kappa(\frp) \colon \kappa(\frp)^p].
    \end{equation}

    Since Frobenius splittings are preserved by localization, $(\O_\nu)_\frp$ is also Frobenius split. Hence, because $[\Gamma_w \colon p\Gamma_w] < [\Gamma_\nu \colon p\Gamma_\nu]$, we have by induction that
    \begin{equation}
        \label{eq:induction-step}
        [\Gamma_w \colon p\Gamma_w][\kappa(\frp) \colon \kappa(\frp)^p] = [\Gamma_w \colon p\Gamma_w][\kappa_w \colon \kappa_w^p] = [K \colon K^p].
    \end{equation}
    Combining \autoref{eq:p-index}, \autoref{eq:rank-1-step}, and \autoref{eq:induction-step}, we then obtain $[\Gamma_\nu \colon p\Gamma_\nu][\kappa_\nu \colon \kappa_\nu^p] = [K \colon K^p]$,
    as desired.

    We next show the converse. Let $e \coloneqq [\Gamma_\nu \colon p\Gamma_\nu]$ and $f \coloneqq [\kappa_\nu \colon \kappa_\nu^p]$. Suppose that $ef = [K \colon K^p]$. As in \autoref{lem:valuation-minimizer-basis}, fix $a_1, \dots, a_e\in K\setminus\{0\}$ so that $\nu(a_1), \dots, \nu(a_e)$ are representatives for the distinct cosets  of  $\Gamma_\nu/p\Gamma_\nu$. Also, fix $b_1, \dots, b_f \in \O_\nu^\times$ so that their images in $\kappa_\nu$ are a basis for $\kappa_\nu$ over $\kappa_\nu^p$. We can and will assume that $a_1 = b_1 = 1$. As shown in \autoref{lem:valuation-minimizer-basis}, $\{a_ib_j\}_{i,j}$ is then a basis of $K$ over $K^p$. Let $\sigma: K\rightarrow K^p$ be the projection onto the $a_1b_1=1$ summand. That is, 
\[
  \sigma \left(\sum_{i,j}c_{ij}^pa_i b_j\right)=c_{11}^p.
\]
This is a splitting of the inclusion $K^p\hookrightarrow K$. Restrict $\sigma$  to $\O_\nu$.  Since $\sigma(1)=1$, it suffices to show that $
\sigma(\O_\nu)\subset \O_\nu^p = \O_\nu \cap K^p$ in order to conclude that $\O_\nu$ is Frobenius split. 

Let $x\in \O_\nu$ and write it as a $K^p$-combination of the $a_ib_j$, say $x= \sum_{i,j}c_{ij}^pa_i b_j$. Then by \autoref{lem:valuation-minimizer-basis},
\[
\nu(\sigma(x))
=\nu(c_{11}^p)=\nu(c_{11}^pa_1b_1)
\geq
\min\{\nu(c_{ij}^pa_i b_j)\}_{i,j}
=\nu(x) \geq 0,
\]
with the last equality following from the lemma.
So $\sigma(\O_\nu)\subset \O_\nu \cap K^p = \O_\nu^p$.
\end{proof}

One can now immediately finish the proof of \autoref{thm:Frobenius-splitting-Abhyankar-equivalence}.

\begin{proof}[Proof of \autoref{thm:Frobenius-splitting-Abhyankar-equivalence}]
    As discussed after the statement of \autoref{thm:Frobenius-splitting-Abhyankar-equivalence}, it is enough to show \autoref{thm:Frobenius-splitting-Abhyankar-equivalence.i}$\Longleftrightarrow$\autoref{thm:Frobenius-splitting-Abhyankar-equivalence.iii}, which follows by \autoref{thm:F-split-F-finite-no-defect}.
\end{proof}

Using \autoref{thm:F-split-F-finite-no-defect} one can also deduce a non-function field version of \autoref{thm:Frobenius-splitting-Abhyankar-equivalence} assuming the existence of an excellent local center.

\begin{corollary}
    \label{cor:F-split-implies-excellent-center-Abhyankar}
    Let $K$ be an $F$-finite field of characteristic $p > 0$. Let $\nu$ be a valuation of $K$ with valuation ring $(\O_\nu,\fram_\nu,\kappa_\nu)$. Suppose $\nu$ admits an excellent local center with fraction field $K$. Then the following are equivalent:
    \begin{enumerate}
        \item The valuation ring $\O_\nu$ is Frobenius split.
        \item Any excellent local center of $\nu$ with fraction field $K$ is an Abhyankar center of $\nu$.
    \end{enumerate}
\end{corollary}

\begin{proof}
    This follows immediately by \autoref{thm:F-split-F-finite-no-defect} and \autoref{thm:Abhyankar-centers-char-p}.
\end{proof}

\begin{remarks}
    {\*}
    \begin{enumerate}
        \item One can have Frobenius split valuation rings that do not admit any excellent local center. For example, if $V$ is a non-trivial perfect valuation ring (i.e., the Frobenius map is an isomorphism), then $V$ is clearly Frobenius split. However, a non-trivial perfect valuation ring cannot have an excellent local center. In fact, $V$ cannot be centered on \emph{any} Noetherian domain with fraction field $K$ because $K$ is perfect. Indeed, we claim that if $R$ is a Noetherian domain with fraction field $K$, then $R = K$. Suppose not. Then $\dim(R) \geq 1$, so let $\frp \in \Spec(R)$ be a height $1$ prime ideal. By the Krull-Akizuki theorem, the integral closure $R'$ of $R_\frp$ in $K$ is a Dedekind domain \cite[\href{https://stacks.math.columbia.edu/tag/00PG}{Tag~00PG}]{stacks-project}. Then $R' = R' \cap K = R' \cap K^p = (R')^p$, where the last equality follows by normality of $(R')^p$. Thus, $R'$ is a perfect Noetherian domain, and hence, it must be a field. But $\dim(R') = \dim(R_\frp) = 1$, which is impossible.
        \item There are interesting examples of valuation rings that are Frobenius split in a non-function field setting. For example, as a consequence of the Hahn-Banach extension property of continuous linear functionals in non-Archimedean functional analysis, one can show that if $(K,|\cdot|)$ is a spherically complete non-Archimedean valued field of characteristic $p > 0$ with valuation ring $\O_K$, then $\O_K$ is Frobenius split. See \cite[\S 5, Page 2509, (15)]{DattaFrobeniusSplittingValuations}. It would be interesting to find general criteria for Frobenius splitting of valuation rings whose ambient fields are not $F$-finite.
    \end{enumerate}
\end{remarks}

\subsection{A ramification-theoretic criterion for \emph{F}-finiteness} Let $K$ be an $F$-finite field and $\nu$ be a valuation of $K$. In \cite[Theorem~4.3.1]{DattaSmithValuations} the authors had \emph{incorrectly} asserted the equivalence of $\nu/\nu^p$ having no ramification defect and the $F$-finiteness of the valuation ring $\O_\nu$. This error was subsequently identified in \cite{DattaSmithValuationsErratum} and numerous characterizations of $F$-finiteness were then given in \cite{DattaSmithValuationsErratum, DattaFrobeniusSplittingValuations, DattaThesis}; for instance, see \autoref{thm:intro-F-finitness-valuation-rings-function-fields}. The current paper rectifies our mistake by showing that $\nu/\nu^p$ having no ramification defect is equivalent to the Frobenius splitting of $\O_\nu$ rather than its $F$-finiteness. However, thanks to \autoref{thm:finiteness-integral-closure} one can also give a purely ramification-theoretic criterion of $F$-finiteness. This appeared implicitly in the results of \cite{DattaSmithValuationsErratum,DattaFrobeniusSplittingValuations,DattaThesis} but we now state the result explicitly.

\begin{theorem}
    \label{thm:F-finiteness-ramification-theoretic}
    Let $K$ be an $F$-finite field of characteristic $p > 0$. Let $\nu$ be a valuation of $K$ with valuation ring $(\O_\nu,\fram_\nu,\kappa_\nu)$ and value group $\Gamma_\nu$. Let $\nu^p$ denote the restriction of $\nu$ to the subfield $K^p$ of $K$. Then the following are equivalent:
    \begin{enumerate}[(i)]
        \item $\O_\nu$ is $F$-finite.
        \item $\nu/\nu^p$ is defectless and $\epsilon(\nu|\nu^p) = [\Gamma_\nu\colon p\Gamma_\nu]$, where $\epsilon(\nu|\nu^p)$ is the initial index of $\nu/\nu^p$.
    \end{enumerate}
\end{theorem}

\begin{proof}
    Note that the valuation ring of $\nu^p$ is $\O_\nu^p$ and $\O_\nu$ is clearly the integral closure of $\O_\nu^p$ in $K$. Since $\O_\nu$ is a valuation ring, this implies that $\nu$ is the \emph{unique} extension of $\nu^p$ to $K$ (this also follows by the pure inseparability of $K/K^p$). The theorem is now a straightforward application of \autoref{thm:finiteness-integral-closure}.
\end{proof}

\begin{remark}
    The equality $\epsilon(\nu|\nu^p) = [\Gamma_\nu \colon p\Gamma_\nu]$ clearly holds if $\nu$ is discrete, that is, if $\Gamma_\nu \cong \mathbf{Z}$. Thus, for discrete valuations $\nu$ of $F$-finite fields, $F$-finiteness of the DVR $\O_\nu$ is equivalent to $\nu/\nu^p$ having no ramification defect, which, by \autoref{thm:F-split-F-finite-no-defect}, is equivalent to $\O_\nu$ being Frobenius split.
\end{remark}

As a consequence of \autoref{thm:F-finiteness-ramification-theoretic} and \cite{DattaSmithValuationsErratum} (see also \cite[\S~5, (2), (3)]{DattaFrobeniusSplittingValuations}), if $\O_\nu$ is $F$-finite and $\nu$ is non-trivial, then 
\[
\epsilon(\nu|\nu^p) = [\Gamma_\nu \colon p\Gamma_\nu] = 
\begin{cases}
    1 \text{ \hspace{4mm} if } \fram_\nu \text{ is not finitely generated,}\\
    p \text{ \hspace{4mm} if } \fram_\nu \text{ is finitely generated}.
\end{cases}
\]
This observation can be used to deduce the following generalization of \autoref{thm:intro-F-finitness-valuation-rings-function-fields}.

\begin{corollary}\cite[\S5, (7)]{DattaFrobeniusSplittingValuations}
    \label{cor:valuations-with-F-finite-Noetherian-centers}
    Let $K$ be an $F$-finite field of characteristic $p > 0$. Let $\nu$ be a non-trivial valuation of $K$ centered on an excellent (equivalently, $F$-finite Noetherian) local domain $(R,\fram_R,\kappa_R)$ with $\Frac(R) = K$. Then the following are equivalent:
    \begin{enumerate}[(i)]
        \item $\O_\nu$ is $F$-finite.
        \item $\O_\nu$ is a DVR and $R$ is an Abhyankar center of $\O_\nu$.
    \end{enumerate}
\end{corollary}

\section{A Frobenius split DVR of a non-\textnormal{\textit{F}}-finite function field that is not divisorial}
\label{sec:non-F-finite-F-split-DVR-function-field}

Our goal will be to prove that Frobenius split valuation rings of function fields need not arise from Abhyankar valuations when the function field is \emph{not} $F$-finite.

\begin{theorem}
    \label{thm:F-split-non-Abhyankar-excellent-DVR-function-field}
    Let $k=\mathbf{F}_p(u_n \colon n \in \mathbf{Z}_{>0})$, where the $u_n$ are algebraically independent over $\mathbf{F}_p$. Then there exist algebraically independent elements $t, y$ over $k$ and a discrete (i.e., with value group $\mathbf{Z}$) valuation $\nu$ of $k(t,y)/k$ that satisfies the following properties:
    \begin{enumerate}[(i)]
        \item $\O_\nu$ is a Frobenius split DVR.\label{thm:F-split-non-Abhyankar-excellent-DVR-function-field.i}
        \item $\nu$ is not an Abhyankar (equivalently, divisorial) valuation of $k(t,y)/k$.\label{thm:F-split-non-Abhyankar-excellent-DVR-function-field.ii}
        \item $\O_\nu$ is excellent.\label{thm:F-split-non-Abhyankar-excellent-DVR-function-field.iii}
    \end{enumerate}
\end{theorem}

We will prove this theorem via a series of intermediate lemmas and constructions.

Note that for $\mathbf{Z}$-valued valuations of function fields, the Abhyankar property is equivalent to the valuation being divisorial. Let $k$ be as in the statement of \autoref{thm:F-split-non-Abhyankar-excellent-DVR-function-field}, and for the indeterminate $t$, set
\[
  y \coloneqq \sum_{n\geq 1}u_n t^{pn}\in k\llbracket t\rrbracket, \hspace{2mm}
  \text{and} \hspace{2mm} K \coloneqq k(t,y)\subseteq k(\!(t)\!).
\]
Let $\nu$ be the restriction to $K$ of the $t$-adic discrete valuation on $k(\!(t)\!)$ whose valuation ring is $k\llbracket t\rrbracket$, and let
\[
  \O_\nu \coloneqq K\cap k\llbracket t\rrbracket
\]
be the valuation ring of $\nu$. 

\begin{lemma}
    \label{lemma:non-Abhyankar}
$\td{K/k} = 2$ and $\nu$ is \emph{not} an Abhyankar valuation of $K/k$.
\end{lemma}

\begin{proof}
To show $\td{K/k} = 2$, it is enough to show that $y$ is transcendental over $k(t)$. Let $Y$ be an indeterminate and $\pi \colon k[t,Y] \to k(\!(t)\!)$ be the unique $k$-algebra map that sends $t \mapsto t$ and $Y \mapsto y$.
Suppose that $P(t,y) = \pi(P(t,Y))=0$ for some nonzero $P(t,Y)\in k[t,Y]$.
Then there exists an $N\geq 1$ such that all coefficients of $P(t,Y)$ lie in $\mathbf{F}_p(u_1,\ldots,u_N)$.
For each $m>N$, the automorphism of $k$ sending $u_m$ to $u_m+1$
and fixing all other $u_n$ extends coefficientwise to an automorphism $\sigma$
of $k(\!(t)\!)$ fixing $t$. Then 
\[0 = \sigma(\pi(P(t,Y))) = P(t,y+t^{pm}),\]
where the last equality follows because $\sigma$ fixes the coefficients of $P(t,Y)$. 
Thus, for all $m > N$, $y + t^{pm}$ is a root of the nonzero polynomial $P(t,Y) \in k[t,Y] = k[t][Y]$, which is impossible since the $y + t^{pm}$ are all distinct for distinct $m$. Consequently, $\trdeg{K/k} =2$.

Since $t\in K$, the value group of $\nu$ is $\mathbf{Z}$, so $\O_\nu$ is a discrete valuation ring with uniformizer $t$. Since $(k\llbracket t\rrbracket,(t),k)$ dominates the local $k$-algebra $(\O_\nu,\fram_\nu,\kappa_\nu)$, it follows that $\kappa_\nu = k$.
Hence
\[
  \dim_{\mathbf Q}(\mathbf{Q} \otimes_{\mathbf Z} \Gamma_\nu) +\trdeg{\kappa_\nu/k} = 1+0 < 2=\trdeg{K/k}.
\]
Thus $\nu$ is not an Abhyankar valuation of $K/k$.
\end{proof}

\medskip
\noindent\textbf{A Frobenius splitting of $k(\!(t)\!)$.}
The monomials
\[
  u^\alpha=\prod_{n\geq 1}u_n^{\alpha_n},
  \qquad 0\leq\alpha_n<p,
\]
where $\alpha$ has finite support, form a $k^p$-basis of $k$.
In particular, $[k:k^p]=\infty$.
Define the $k^p$-linear projection
\[
  \phi \colon k\longrightarrow k^p
\]
by $\phi(1)=1$ and $\phi(u^\alpha)=0$ whenever $\alpha$ is not a tuple of all zeros. Extend $\phi$ to a $k^p(\!(t^p)\!)$-linear map 
\begin{align*}
    \Phi \colon k(\!(t)\!)&\longrightarrow k^p(\!(t^p)\!)\\
    \sum_{m\geq m_0}a_m t^m &\longmapsto \sum_{\substack{m\geq m_0\\p\mid m}}\phi(a_m)t^m.
\end{align*}
In other words, $\Phi$ is the composition $k(\!(t)\!) \to k^p(\!(t)\!) \to k^p(\!(t^p)\!)$, where the first map just applies $\phi$ to the coefficients of the Laurent series and the second map is projection onto the basis element $1$ for the canonical basis $\{1,t,\dots,t^{p-1}\}$ of $k^p(\!(t)\!)$ over $k^p(\!(t^p)\!)$.

Let $\nu_t$ be the $t$-adic valuation on $k(\!(t)\!)$. Then it is also clear that for all $z\in k(\!(t)\!)^\times$, we have $\Phi(z) = 0$ or
\begin{equation}\label{eq:valuation-nondecreasing}
  \nu_t(\Phi(z))\geq \nu_t(z).
\end{equation}
Moreover, $\Phi(1) = 1$, so $\Phi$ is a Frobenius splitting of $k(\!(t)\!)$.

\begin{lemma}
    \label{lem:Induced-splitting}
    The map $\Phi$ restricts to a Frobenius splitting of $K$, i.e., $\Phi(K) \subseteq K^p$.
\end{lemma}

\begin{proof}
    A basis of $K$ over $K^p$ is given by the monomials of the form $u^\alpha t^iy^j$ for a multi-index $\alpha$ as above and $0\leq i,j<p$. It is enough to show that $\Phi(u^\alpha t^iy^j)\in K^p$ for all such monomials. Suppose $j \geq 1$. Since $y = \sum_{n \geq 1} u_nt^{pn}$, the exponent of $t$ in every nonzero term of $y^j$ is divisible by $p$. Thus, if $i \geq 1$, then every nonzero term of $u^\alpha t^i y^j$ has an exponent of $t$ that is not divisible by $p$. Hence, $\Phi(u^\alpha t^i y^j) = 0 \in K^p$ when $i, j \geq 1$. If $j = 0$, then
    \[
    \Phi(u^\alpha t^i y^j) =
    \begin{cases}
          0 \hspace{11mm}\text{ if } i \geq 1,\\
          \phi(u^\alpha) \hspace{3mm}\text{ if } i = 0.
    \end{cases}
    \]
    Consequently, $\Phi(u^\alpha t^i y^j) \in K^p$ when $j = 0$. It remains to check the case where $j \geq 1$ and $i = 0$. Then $u^\alpha t^i y^j = u^\alpha y^j$. Suppose 
    \[u^\alpha = u_1^{\alpha_1}\cdots u_N^{\alpha_N}.\]
    Write $y = y_N + y'$, where $y_N = \sum_{n=1}^N u_n t^{pn}$ and $y' = \sum_{n>N} u_n t^{pn}$. Then 
    \[u^\alpha y^j = \sum_{m=0}^j \binom{j}{m} u^\alpha y_N^m (y')^{j-m}.\]
    For $j - m > 0$, each coefficient of the power series $u^\alpha y^m_N (y')^{j-m}$ is a finite sum of $\mathbf{F}_p$ multiples of monomials, each containing a factor $u_n^e$ with $n > N$ and $1 \leq e \leq j-m < p$. Thus $\phi$ kills each summand in each coefficient, so $\Phi$ kills the power series. Consequently,
    \[
    \Phi(u^\alpha y^j) = \Phi(u^\alpha y_N^j).
    \]
    Since $u^\alpha y_N^j \in k[t^p]$, it is then clear by the definition of $\Phi$ that $\Phi(u^\alpha y_N^j) \in k^p[t^p] \subset K^p$.
\end{proof}

We can now prove \autoref{thm:F-split-non-Abhyankar-excellent-DVR-function-field}.

\begin{proof}[Proof of \autoref{thm:F-split-non-Abhyankar-excellent-DVR-function-field}]
    We choose $t,y$ over $k$ as above and $\nu$ to be the restriction to $k(t,y)$ of the $t$-adic valuation on $k(\!(t)\!)$. By \autoref{lemma:non-Abhyankar} we know that $\nu$ is not Abhyankar, showing that condition \autoref{thm:F-split-non-Abhyankar-excellent-DVR-function-field.ii} is satisfied. In order to show that $\O_\nu$ is Frobenius split, i.e., condition \autoref{thm:F-split-non-Abhyankar-excellent-DVR-function-field.i} is satisfied, by \autoref{lem:Induced-splitting} it is enough to show that $\Phi(\O_\nu) \subseteq \O_\nu^p$. Since $\O_\nu = k\llbracket t \rrbracket \cap K$ and $\O_\nu^p = k^p\llbracket t^p\rrbracket \cap K^p$, we have $\Phi(\O_\nu) \subseteq \Phi(k\llbracket t \rrbracket) \cap \Phi(K) \subseteq k^p\llbracket t^p\rrbracket \cap K^p = \O_\nu^p$. Here $\Phi(k\llbracket t \rrbracket) \subseteq k^p\llbracket t^p\rrbracket$ by \autoref{eq:valuation-nondecreasing}, while $\Phi(K) \subseteq K^p$ by \autoref{lem:Induced-splitting}. Therefore, $\O_\nu$ is Frobenius split. Finally, \autoref{thm:F-split-non-Abhyankar-excellent-DVR-function-field.iii} is satisfied because any Frobenius split DVR is excellent by \cite[Corollary~7.4.3]{DattaMurayamaFsolidity}.
\end{proof}

\begin{remark}
    Let $k$ be an $F$-finite field of characteristic $p > 0$, and consider the Laurent series field $k(\!(t)\!)$ with the $t$-adic valuation. Choose $y \in tk\llbracket t \rrbracket$ such that $\{t,y\}$ is algebraically independent over $k$. Let $\nu$ be the restriction of the $t$-adic valuation on $k(\!(t)\!)$ to the subfield $K = k(t,y)$. The authors showed in \cite{DattaSmithExcellence} that $\O_\nu$ is a non-excellent DVR with no nonzero $\O_\nu^p$-linear maps $\O_\nu \to \O_\nu^p$. Thus, the above example crucially uses that the ground field is not $F$-finite.
\end{remark}

\section{Acknowledgements}
The question of when valuation rings are Frobenius split was posed to the second author by Karl Schwede and Zsolt Patakfalvi. We thank them, especially Karl Schwede, for helpful discussions over the years and their interest. The first author proposed the problem of proving that Abhyankar valuations are the only valuations of function fields of $F$-finite surfaces that have Frobenius split valuation rings in the 2022 Michigan Research Experience for Graduates (MREG) for a group of four graduate students -- Sandra Nair, Sandra Sandoval-G{\'o}mez, Dawei Shen and Ying Wang. The first author thanks the organizers of MREG for creating a stimulating virtual environment for the various research groups. After the conclusion of the week-long program, the first author, Nair, Sandoval-G{\'o}mez and Wang continued working on the surface case of the problem without success. The first author would consequently like to express his gratitude to the three of them for thinking about this problem. Both authors would also like to thank them for giving us permission to publish this paper, which relies on methods different from the ones that were discussed during and after MREG with Nair, Sandoval-G{\'o}mez and Wang. The first author was partially supported by NSF-DMS grant number 2502333 and Simons Foundation grant number MP-TSM-00002400. The second author was partially supported by NSF-DMS grant numbers 2101075 and 2200501.

\bibliographystyle{skalpha}
\bibliography{main,preprints}

@preamble{"\def\cfudot#1{\ifmmode\setbox7\hbox{$\accent"5E#1$}\else \setbox7\hbox{\accent"5E#1}\penalty 10000\relax\fi\raise 1\ht7 \hbox{\raise.1ex\hbox to 1\wd7{\hss.\hss}}\penalty 10000 \hskip-1\wd7\penalty 10000\box7} "}

@incollection{VaquieValuations,
    author    = {Vaqui{\'e}, Michel},
    title     = {Valuations and local uniformization},
    booktitle = {Singularity Theory and Its Applications},
    series    = {Advanced Studies in Pure Mathematics},
    volume    = {43},
    pages     = {477--527},
    publisher = {Mathematical Society of Japan},
    address   = {Tokyo},
    year      = {2006},
    doi       = {10.2969/aspm/04310477},
    mrnumber  = {MR2325152},
}

@article {JonssonMustataValuations,
    AUTHOR = {Jonsson, Mattias and Musta{\c{t}}{\u{a}}, Mircea},
     TITLE = {Valuations and asymptotic invariants for sequences of ideals},
   JOURNAL = {Ann. Inst. Fourier (Grenoble)},
  FJOURNAL = {Annales de l'Institut Fourier},
    VOLUME = {62},
      YEAR = {2012},
    NUMBER = {6},
     PAGES = {2145--2209},
  MRNUMBER = {MR3060755},
       DOI = {10.5802/aif.2746},
       URL = {https://doi.org/10.5802/aif.2746},
}

@article {AbhyankarLocalDomain,
    AUTHOR = {Abhyankar, Shreeram},
     TITLE = {On the valuations centered in a local domain},
   JOURNAL = {Amer. J. Math.},
  FJOURNAL = {American Journal of Mathematics},
    VOLUME = {78},
      YEAR = {1956},
    NUMBER = {2},
     PAGES = {321--348},
  MRNUMBER = {MR0082477},
       DOI = {10.2307/2372519},
       URL = {https://doi.org/10.2307/2372519},
}

@article {CutkoskyAbhyankar,
    AUTHOR = {Cutkosky, Steven Dale},
     TITLE = {Local uniformization of {A}bhyankar valuations},
   JOURNAL = {Michigan Math. J.},
  FJOURNAL = {Michigan Mathematical Journal},
    VOLUME = {71},
      YEAR = {2022},
    NUMBER = {4},
     PAGES = {859--891},
      ISSN = {0026-2285,1945-2365},
   MRCLASS = {13A18 (12J10)},
  MRNUMBER = {MR4505369},
       DOI = {10.1307/mmj/20205888},
       URL = {https://doi.org/10.1307/mmj/20205888},
}

@phdthesis {DattaThesis,
    AUTHOR = {Datta, Rankeya},
     TITLE = {A tale of valuation rings in prime characteristic},
    SCHOOL = {University of Michigan},
      YEAR = {2018},
       URL = {https://hdl.handle.net/2027.42/146066},
}

@article {DattaFrobeniusSplittingValuations,
    AUTHOR = {Datta, Rankeya},
     TITLE = {Frobenius splitting of valuation rings and {$F$}-singularities of centers},
   JOURNAL = {Algebra Number Theory},
    VOLUME = {15},
      YEAR = {2021},
    NUMBER = {10},
     PAGES = {2485--2512},
  MRNUMBER = {MR4377857},
       DOI = {10.2140/ant.2021.15.2485},
       URL = {https://doi.org/10.2140/ant.2021.15.2485},
}

@incollection {DattaSmithExcellence,
    AUTHOR = {Datta, Rankeya and Smith, Karen E.},
     TITLE = {Excellence in prime characteristic},
 BOOKTITLE = {Local and global methods in algebraic geometry},
    SERIES = {Contemp. Math.},
    VOLUME = {712},
     PAGES = {105--116},
 PUBLISHER = {Amer. Math. Soc., [Providence], RI},
      YEAR = {[2018] \copyright 2018},
   MRCLASS = {13A35},
  MRNUMBER = {3832401},
MRREVIEWER = {Yongwei Yao},
       DOI = {10.1090/conm/712/14344},
       URL = {https://doi-org.proxy.cc.uic.edu/10.1090/conm/712/14344},
}

@article{DattaTuckerSplinterLoci,
    AUTHOR = {Datta, Rankeya and Tucker, Kevin},
     TITLE = {Openness of splinter loci in prime characteristic},
   JOURNAL = {J. Algebra},
  FJOURNAL = {Journal of Algebra},
    VOLUME = {629},
      YEAR = {2023},
     PAGES = {307--357},
  MRNUMBER = {MR4583731},
       DOI = {10.1016/j.jalgebra.2023.03.025},
       URL = {https://doi.org/10.1016/j.jalgebra.2023.03.025},
}

@article {KnafKuhlmannAbhyankar,
    AUTHOR = {Knaf, Hagen and Kuhlmann, Franz-Viktor},
     TITLE = {Abhyankar places admit local uniformization in any
              characteristic},
   JOURNAL = {Ann. Sci. \'{E}cole Norm. Sup. (4)},
  FJOURNAL = {Annales Scientifiques de l'\'{E}cole Normale Sup\'{e}rieure. Quatri\`eme
              S\'{e}rie},
    VOLUME = {38},
      YEAR = {2005},
    NUMBER = {6},
     PAGES = {833--846},
      ISSN = {0012-9593},
   MRCLASS = {13A18 (12J10)},
  MRNUMBER = {MR2216832},
MRREVIEWER = {Vasil\cprime  \={I}. Andr\={\i}\u{\i}chuk},
       DOI = {10.1016/j.ansens.2005.09.001},
       URL = {https://doi-org.proxy2.cl.msu.edu/10.1016/j.ansens.2005.09.001},
}

@article{DattaMurayamaFsolidity,
    author = {Datta, Rankeya and Murayama, Takumi},
    note = {(with an appendix by Karen E. Smith) arXiv:2007.10383},
    title = {Excellence, {$F$}-singularities, and solidity},
    year = {2020}
}

@article {DattaSmithValuations,
    AUTHOR = {Datta, Rankeya and Smith, Karen E.},
     TITLE = {Frobenius and valuation rings},
   JOURNAL = {Algebra Number Theory},
  FJOURNAL = {Algebra \& Number Theory},
    VOLUME = {10},
      YEAR = {2016},
    NUMBER = {5},
     PAGES = {1057--1090},
      ISSN = {1937-0652},
   MRCLASS = {13A35 (13F30 14B05)},
  MRNUMBER = {MR3531362},
MRREVIEWER = {Ana Bravo},
       DOI = {10.2140/ant.2016.10.1057},
       URL = {https://doi-org.proxy2.cl.msu.edu/10.2140/ant.2016.10.1057},
}

@book{BourbakiCA,
	Address = {Berlin},
	Author = {Bourbaki, Nicolas},
	Isbn = {3-540-64239-0},
	Mrclass = {13-XX (00A05)},
	Mrnumber = {MR1727221},
	Note = {Translated from the French, Reprint of the 1989 English translation},
	Pages = {xxiv+625},
	Publisher = {Springer-Verlag},
	Series = {Elements of Mathematics (Berlin)},
	Title = {Commutative algebra. {C}hapters 1--7},
	Year = {1998}
}

@article{KunzOnNoetherianRingsOfCharP,
	Author = {Kunz, Ernst},
	Fjournal = {American Journal of Mathematics},
	Issn = {0002-9327},
	Journal = {Amer. J. Math.},
	Mrclass = {13E05 (14E15)},
	Mrnumber = {MR0432625 (55 \#5612)},
	Mrreviewer = {Hamet Seydi},
	Number = {4},
	Pages = {999--1013},
	Title = {On {N}oetherian rings of characteristic {$p$}},
	Volume = {98},
	Year = {1976}
}

@article {DattaSmithValuationsErratum,
    AUTHOR = {Datta, Rankeya and Smith, Karen E.},
     TITLE = {Correction to the article ``{F}robenius and valuation rings''
              [ {MR}3531362]},
   JOURNAL = {Algebra Number Theory},
  FJOURNAL = {Algebra \& Number Theory},
    VOLUME = {11},
      YEAR = {2017},
    NUMBER = {4},
     PAGES = {1003--1007},
      ISSN = {1937-0652},
   MRCLASS = {13A35 (13F30 14B05)},
  MRNUMBER = {3665644},
       DOI = {10.2140/ant.2017.11.1003},
       URL = {https://doi-org.proxy2.cl.msu.edu/10.2140/ant.2017.11.1003},
}

@article {XuQuasiMonomial,
    AUTHOR = {Xu, Chenyang},
     TITLE = {A minimizing valuation is quasi-monomial},
   JOURNAL = {Ann. of Math. (2)},
  FJOURNAL = {Annals of Mathematics. Second Series},
    VOLUME = {191},
      YEAR = {2020},
    NUMBER = {3},
     PAGES = {1003--1030},
      ISSN = {0003-486X},
   MRCLASS = {14E30 (14J17 14J45)},
  MRNUMBER = {4088355},
MRREVIEWER = {Yuchen Liu},
       DOI = {10.4007/annals.2020.191.3.6},
       URL = {https://doi.org/10.4007/annals.2020.191.3.6},
}

@incollection {KuhlmannDefect,
    AUTHOR = {Kuhlmann, Franz-Viktor},
     TITLE = {The defect},
 BOOKTITLE = {Commutative algebra---{N}oetherian and non-{N}oetherian
              perspectives},
     PAGES = {277--318},
 PUBLISHER = {Springer, New York},
      YEAR = {2011},
   MRCLASS = {13A18},
  MRNUMBER = {2762515},
MRREVIEWER = {Carlos Galindo},
       DOI = {10.1007/978-1-4419-6990-3\_11},
       URL = {https://doi-org.proxy2.cl.msu.edu/10.1007/978-1-4419-6990-3_11},
}

@misc{stacks-project,
	Author = {The {Stacks Project Authors}},
	Howpublished = {\url{http://stacks.math.columbia.edu}},
	Shorthand = {Sta},
	Title = {{\itshape Stacks Project}}
}

\end{document}